\documentclass[]{interact}
\usepackage[utf8]{inputenc}

\usepackage{epstopdf}
\usepackage{subfigure}%
\usepackage[caption=false]{subfig}
\usepackage{algorithm} 
\usepackage{algorithmic} 

\usepackage[numbers,sort&compress]{natbib}
\bibpunct[, ]{[}{]}{,}{n}{,}{,}
\makeatletter
\def\NAT@def@citea{\def\@citea{\NAT@separator}}
\makeatother

\theoremstyle{plain}
\newtheorem{theorem}{Theorem}[section]
\newtheorem{lemma}[theorem]{Lemma}

\newtheorem{proposition}[theorem]{Proposition}

\theoremstyle{definition}
\newtheorem{definition}[theorem]{Definition}

\theoremstyle{remark}

\begin{document}


\title{Convergence analysis on the alternating direction method of multipliers for the cosparse optimization problem}

\author{
\name{Zisheng Liu
\textsuperscript{a}\thanks{CONTACT Zisheng Liu Email: Z. Liu Email: liuzisheng0710@163.com} and
Ting Zhang\textsuperscript{~b}}
\affil{\textsuperscript{a}School of Statistics and Big Data, Henan University of Economics and Law, Zhengzhou, China;\\
        \textsuperscript{b}School of Mathematics and Information Science, Henan University of Economics and Law, Zhengzhou, China.}
}

\maketitle

\begin{abstract}
From a dual perspective of the sparse representation model, Nam et al. proposed the cosparse analysis model. In this paper, we aim to investigate the convergence of the alternating direction method of multipliers (ADMM) for the cosparse optimization problem. First, we examine the variational inequality representation of the cosparse optimization problem by introducing auxiliary variables. Second, ADMM is used to solve cosparse optimization problem. Finally, by utilizing a tight frame with a uniform row norm and building upon lemmas and the strict contraction theorem, we establish a worst-case $\mathcal{O}(1/t)$ convergence rate in the ergodic sense.
\end{abstract}


\begin{amscode}
90C25 
\end{amscode}

\section{Introduction}
\label{sec:introduction}
Low-dimensional signal recovery takes advantage of the inherent low-dimensionality of many natural signals, despite their high ambient dimension.
Utilizing prior information about the low-dimensional space can significantly aid in recovering the signal of interest.
Sparsity, a widely recognized form of prior information, serves as the foundation for the burgeoning field of compressive sensing (CS \cite{Donoho1,Donoho2,CS1,CS2,CS3}).
The recovery of sparse inputs has found numerous applications in areas such as imaging, speech, radar signal processing, sub-Nyquist sampling, and more \cite{c1,c2,c3,c4}.
A typical sparse recovery problem is associated with the following linear system:
\begin{eqnarray}\label{e1}
y=Mx,
\end{eqnarray}
where $y\in R^m$ is an observed vector, $M\in R^{m\times d}$ is a measurement matrix and $x\in R^d$ is an unknown signal which would be estimated from $y$. According to the Nyquist-Shannon sampling theorem, if the $k$-space data is undersampled so much that it fails to meet the Nyquist sampling criterion, then reconstructing the data can be difficult or impossible without prior knowledge of $x$.

\subsection{Sparse synthesis model}
Over the past decade, the application of compressed sensing significantly increased the image reconstruction speed and efficiency because of its capability to reconstruct images from highly undersampled signals. Sparse prior is widely used in CS-based reconstruction methods. For the sparse synthesis model, if a vector $x$ is sufficient sparse, under the incoherence assumptions on the measurement matrix $M$, $x$ can be robustly estimated by the problem
\begin{equation}\label{e2}
\begin{aligned}
\min_x&~ \|x\|_\tau\\
s.t.&~y = Mx ,
\end{aligned}
\end{equation}
where $0 \leq\tau \leq 1$.
The advanced ideas and methods have been explored by applications in signals and image processing \cite{c19,c20,c21,c22}. After years of research, this model is becoming more and more mature and stable.

\subsection{Cosparse analysis model}
In the recent decade, the cosparse analysis model is an alternative approach has gained popularity \cite{c23,c5,c30,c62,Song,Davies1}.
Within this framework, a potentially redundant analysis operator $D\in\mathbb{R}^{n\times d}(n\geq d)$ is employed, and the analyzed vector $D x$ is expected to be sparse.
This implies that a signal $x\in R^d$ belongs to the cosparse analysis model with cosparsity $\ell$ if $\ell=n-\|D x\|_0$.
In this paper, the quantity $\ell$ represents the number of rows in $D$ that are orthogonal to the signal.
Consequently, $x$ is referred to as $\ell$-cosparse or simply cosparse.
The specific definitions of cosparse and cosupport can be found in literature \cite{c5}, for ease of reference, we have listed them below.
\begin{definition}[Cosparse]\label{d1}
A signal $x \in R^d$ is said to
be cosparse with respect to an analysis operator $D\in R^{n\times d}$ if the analysis representation vector $D x$ contains many zero elements. Further, the number of zero elements
\begin{eqnarray*}
\ell= n-\|D x\|_0
\end{eqnarray*}
is called the cosparsity of $x$, we also say $x$ is $\ell$-cosparse.
\end{definition}

\begin{definition}[Cosupport]\label{d2}
For a signal $x \in R^d$ and a given analysis operator $D\in R^{n\times d}$ with its rows $D_j\in R^d(1\leq j\leq n)$, the cosupport is defined by
\begin{eqnarray*}
\Lambda:=\{j|\langle D_j,x\rangle=0\}.
\end{eqnarray*}
\end{definition}
In this paper, $D$ is a tight frame with uniform row norm. We remind the reader that
a frame is defined as below.

\begin{definition}[Frame\cite{c6,c7}]\label{d3}
Let $\Phi=\{\varphi_{i}\}_{i=1}^{N} \subseteq R^{n}$ be a vector sequence of the Hilbert space with $N\geq n$. If there exist constants $0<A\leq B<\infty$ such that
\begin{eqnarray}
\forall x\in R^{n},~~~A\|x\|^{2}\leq\sum_{i=1}^N|\langle x,\varphi_{i}\rangle|^2\leq B\|x\|^{2},
\end{eqnarray}
then $\Phi$ is referred to as a finite frame of $R^{n}$. The constants $A$ and $B$ in the above formula are known as the lower and upper bounds of the finite frame $\Phi$, respectively. They are considered to be the optimal bounds, with $A$ being the supremum in the lower bound and $B$ being the infimum in the upper bound. If $A=B$, then the frame $\Phi$ is called an $A$-tight frame. If $A=B=1$, then $\Phi$ is called a Parseval frame. If there exists a constant $C$ such that each meta-norm $\|\varphi_{i}\|=C$ of the frame $\Phi$, then $\Phi$ is called an iso-norm frame. In particular, for a tight frame, if $C=1$, it is referred to as a uniformly tight frame.
\end{definition}

According to the definition of cosparsity, the cosparse analysis model focuses on the zero elements of the analysis representation vector $D x$, rather than the non-zero elements. This perspective contrasts with the sparse synthesis model. If the cosparsity $\ell$ is significantly large, meaning that the number of zeros $\ell$ is close to $d$, we say that $x$ has a cosparse representation. The cosupport set is identified by iteratively removing rows from $D$ for which $\langle D_j,x\rangle\neq0$ until the index set $\Lambda$ remains unchanged, with $|\Lambda|\geq\ell$.

If the analysis representation vector $Dx$ is sparse, similar to the sparse model, the estimation of $x$ from the measurements can be achieved by
\begin{equation}\label{e3}
\begin{aligned}
\min_x &~\|D x\|_0\\
s.t.&~y = Mx.
\end{aligned}
\end{equation}
The minimization problem \eqref{e3} is known to be NP-hard \cite{c5}, necessitating the use of approximation methods.
Similar to the sparse model, one option is to use the greedy analysis pursuit (GAP) approach, which is inspired by the orthogonal matching pursuit (OMP) algorithm \cite{c23,c5,Giryes}.
Alternatively, the nonconvex $\ell_0$ norm can be approximated by the convex $\ell_1$ norm, leading to the relaxed problem known as analysis basis pursuit (ABP) \cite{c65}. In this case, $x$ can be estimated by solving a modified optimization problem
\begin{equation}\label{e8}
\begin{aligned}
&\min_x \|D x\|_1 \\
&s.t.~\|y - Mx\|_2\leq\epsilon,
\end{aligned}
\end{equation}
where $\|\cdot\|_1$ is the $\ell_1$ norm that sums the absolute values of a vector and $\epsilon$ is a upper bound on the noise level $\|v\|_2$.

ABP is equivalent to the unconstrained optimization
\begin{equation}\label{e9}
\begin{aligned}
\min_x \|D x\|_1 + \frac{\alpha}{2}\|y - Mx\|^2_2,
\end{aligned}
\end{equation}
which we call analysis LASSO (ALASSO).
It can be said that ABP and ALASSO are equivalent in the sense that for any $\epsilon>0$, there exists an $\alpha$ such that the optimal solutions of ABP and ALASSO are identical. For the optimization problem \eqref{e9}, our previous work presented the modified GAP algorithm and error analysis \cite{c47,c62}. The simulations we conducted demonstrated the advantages of the proposed method for the cosparse optimization problem. These optimization problems can also be solved using interior point methods \cite{c48}. However, as the problem dimension increases, these techniques become time-consuming since they require solutions of linear systems. Other suggested approaches include the alternating direction method of multipliers (ADMM) \cite{c63,c64,Han,Han1} and the accelerated alternating minimization method (AAM) \cite{c66}. In this paper, we propose a new way to analyze the convergence theory of the cosparse optimization problem based on the variational inequality.

\subsection{Organization of the paper}
Our focus in this paper is on the cosparse optimization problem and its convergence study based on a variational inequality. The paper is structured as follows: In Section 2, we introduce auxiliary variables and investigate the variational inequality characterization of the cosparse optimization problem. In Section 3, we present several lemmas that establish the strict contraction of the ADMM for the cosparse optimization problem. Using these lemmas and the strict contraction theorem, we establish a worst-case $\mathcal{O}(1/t)$ convergence rate in the ergodic sense. Finally, Section 4 provides a brief conclusion.

\section{Preliminaries}

To apply the ADMM for solving the cosparse optimization problem (\ref{e9}), we convert the unconstrained optimization problem mentioned above into a constrained optimization problem as follows
\begin{equation}\label{e10}
\begin{aligned}
\min_{x,z} & \|z\|_1 + \frac{\alpha}{2}\|y - Mx\|^2_2\\
s.t. & ~Dx - z=0,
\end{aligned}
\end{equation}
where an auxiliary variable $z\in R^n$ is introduced in (\ref{e9}) to transfer $Dx$ out of the nondifferentiable term $\|\cdot\|_1$ and $\alpha>0$ is a penalty parameter.

In this section, we summarize the variational inequality (VI) characterization of \eqref{e10}. Initially, we present the optimality condition of the constrained optimization problem (\ref{e10}), which forms the foundation for our subsequent convergence analysis \cite{c54,Huang1}. We then proceed to express the Lagrangian function of (\ref{e10}) as follows

\begin{eqnarray}\label{e12}
L(z,x,\lambda) = \|z\|_1 + \frac{\alpha}{2}\|y - Mx\|^2_2 - \lambda^T(Dx - z).
\end{eqnarray}
In (\ref{e12}), we assume that $x\in\mathcal{X}$, $z\in\mathcal{Z}$ and $\lambda\in R^n$ where $\mathcal{X}\subset R^d$ and $\mathcal{Z}\subset R^n$ are closed convex sets, we call ($z^*,x^*,\lambda^*)\in\Omega:=\mathcal{Z\times X }\times R^n$ to be a saddle point of $L(z,x,\lambda)$ if the following inequalities are satisfied
\begin{eqnarray}\label{e13}
\begin{aligned}
L(z^*,x^*,\lambda) &\leq L(z^*,x^*,\lambda^*)\leq L(z,x,\lambda^*).
\end{aligned}
\end{eqnarray}
Obviously, a saddle point ($z^*,x^*,\lambda^*$) can be characterized by the system
\begin{eqnarray}\label{e14}
\left\{
      \begin{array}{lll}
      z^* = \arg\min\{L(z,x^*,\lambda^*)|z\in\mathcal{Z}\},\\
      x^* = \arg\min\{L(z^*,x,\lambda^*)|x\in\mathcal{X}\},\\
      \lambda^* =\arg\max\{L(z^*,x^*,\lambda)|\lambda\in R^n\},
      \end{array}
\right.
\end{eqnarray}
which can be rewritten as
\begin{eqnarray}\label{e15}
\begin{aligned}
\left\{
      \begin{array}{lll}
      z^*\in\mathcal{Z}, L(z,x^*,\lambda^*) - L(z^*,x^*,\lambda^*)\geq 0,\\
      x^*\in\mathcal{X}, L(z^*,x,\lambda^*) - L(z^*,x^*,\lambda^*)\geq 0,\\
      \lambda^*\in R^n, L(z^*,x^*,\lambda^*) - L(z^*,x^*,\lambda)\geq 0.
      \end{array}
\right.
\end{aligned}
\end{eqnarray}

Below, we present a summary of the method for expressing the optimality condition of the cosparse analysis model \eqref{e10} via a variational inequality.

\begin{proposition}\label{P1}
Suppose $\mathcal{X}\subset R^d$ is a closed convex set, and $\theta(x):R^d\rightarrow R$ is a convex function. Furthermore, let $f(x)$ be differentiable in $\mathcal{X}$. We assume that the set of solutions for the minimization problem $\min\{\theta(x) + f(x)|x\in\mathcal{X}\}$ is nonempty, then,
\begin{eqnarray}\label{e16}
x^*=\arg\min\{\theta(x) + f(x)|x\in\mathcal{X}\}
\end{eqnarray}
if and only if
\begin{eqnarray}\label{e17}
x,~x^*\in\mathcal{X}, ~\theta(x) - \theta(x^*) + (x - x^*)^T \nabla f(x^*)\geq 0.
\end{eqnarray}
\end{proposition}

The proof of Proposition \ref{P1} is available in \cite{c61}. Let $\theta_1(z)=\|z\|_1$ and $\theta_2(x) = \frac{\alpha}{2}\|y - Mx\|^2_2$, according to the above inequality \eqref{e17}, a saddle point ($z^*,x^*,\lambda^*$) of the Lagrangian function (\ref{e12}) can be characterized by a solution point of the following variational inequality
\begin{eqnarray}\label{e19}
\begin{aligned}
& \omega, \omega^*\in\Omega,~~\theta(u) - \theta(u^*) + (\omega - \omega^*)^T F(\omega^*)\geq 0,
\end{aligned}
\end{eqnarray}
where
\begin{eqnarray}\label{e21}
\theta(u) = \theta_1(z) + \theta_2(x),~~ \Omega = \mathcal{Z\times X }\times R^n,
\end{eqnarray}
and
\begin{eqnarray}\label{e20}
\begin{aligned}
&\omega=\left(
         \begin{array}{c}
           z \\
           x \\
           \lambda \\
         \end{array}
       \right),
~~u=\left(
    \begin{array}{c}
      z \\
      x \\
    \end{array}
  \right),
~~F(\omega) = \left(
                \begin{array}{c}
                  \lambda \\
                  -D^T \lambda \\
                  Dx - z \\
                \end{array}
              \right),
              \end{aligned}
\end{eqnarray}
Since $F$ is an affine operator, and
\begin{eqnarray}\label{e21}
F(\omega) = \left(
                          \begin{array}{ccc}
                            0   & 0   & I \\
                            0   & 0   & -D^T \\
                            -I  & D  & 0 \\
                          \end{array}
                        \right)\left(
                                 \begin{array}{c}
                                   z \\
                                   x \\
                                   \lambda \\
                                 \end{array}
                               \right),
\end{eqnarray}
According to the antisymmetry of the affine matrix, it follows that
\begin{eqnarray}\label{e36}
(\omega - \bar{\omega})^T [F(\omega) - F(\bar{\omega})]\equiv 0, ~\forall~ \omega,\bar{\omega}\in \Omega.
\end{eqnarray}
Using inequality (\ref{e17}) and combining (\ref{e12}), we derive the following conclusion with $(z^*, x^*,\lambda^*)\in\Omega$,
\begin{eqnarray}\label{e18}
\left\{
      \begin{array}{lll}
       \theta_1(z) - \theta_1(z^*) + (z - z^*)^T \lambda^*\geq 0,\\
       \theta_2(x) - \theta_2(x^*) + (x - x^*)^T (-D^T \lambda^*)\geq 0,\\
       (\lambda - \lambda^*)^T(Dx^* - z^*)\geq 0.
      \end{array}
\right.
\end{eqnarray}
After conducting the aforementioned analysis, the linear constrained cosparse optimization problem is reformulated as a variational inequality. Consequently, the task is ultimately simplified to identifying a saddle point of the Lagrangian function. In the subsequent section, the convergence analysis of the ADMM method for addressing the cosparse optimization problem, as denoted by equation \eqref{e10}, will be discussed.

\section{Convergence analysis of the cosparse optimization problem}

\subsection{Variational inequality characterization of ADMM}

The augmented Lagrangian function of the problem (\ref{e10}) can be formulated as follows
\begin{eqnarray}\label{e22}
\begin{aligned}
\mathcal{L}_{\beta}(z,x,\lambda) =& \|z\|_1 + \frac{\alpha}{2}\|y - Mx\|^2_2 - \lambda^T(Dx - z) + \frac{\beta}{2}\|Dx - z\|_2^2,
\end{aligned}
\end{eqnarray}
where $\lambda$ is the Lagrange multiplier and $\beta>0$ is a penalty parameter for the linear constraints. Thus, applying directly the augmented Lagrangian function \eqref{e22} and starting with an initial iterate $(x^0,\lambda^0)\in\mathcal{ X }\times R^n$, the ADMM generates its sequence via following iterative scheme
\begin{eqnarray}\label{e24}
\left\{
      \begin{array}{lll}
      z^{k+1} = \arg\min\{\mathcal{L}_{\beta}(z,x^{k},\lambda^{k})|z\in\mathcal{Z}\},\\
      x^{k+1} = \arg\min\{\mathcal{L}_{\beta}(z^{k+1},x,\lambda^{k})|x\in\mathcal{X}\},\\
      \lambda^{k+1} = \lambda^{k} - \beta(Dx^{k+1} - z^{k+1}),~\lambda\in R^n
      \end{array}
\right.
\end{eqnarray}
the corresponding variational inequalities of (\ref{e24}) can be given as
\begin{eqnarray}\label{e25}
\begin{aligned}
\left\{
      \begin{array}{lll}
      \theta_1(z) - \theta_1(z^{k+1}) + (z - z^{k+1})^T [\lambda^k -   \beta(Dx^k - z^{k+1})]\geq 0,\\
      \theta_2(x) - \theta_2(x^{k+1}) + (x - x^{k+1})^T [-D^T \lambda^k +  \beta D^T(Dx^{k+1} - z^{k+1})]\geq 0,\\
      (\lambda - \lambda^{k+1})^T [(Dx^{k+1} - z^{k+1})  + \frac{1}{\beta}(\lambda^{k+1} - \lambda^{k})]\geq 0.
      \end{array}
\right.
\end{aligned}
\end{eqnarray}
For some reviews on the classical ADMM, one can refer to literatures \cite{Han,c55,c56,c59,c54,Huang2}.

\subsection{Assertions}
To establish that $\{\omega^k\}$ is strictly contractive with respect $\Omega$, we first present several lemmas.

\begin{lemma}\label{L1}
Let the sequence $\{\omega^k\}$ be generated by (\ref{e24}). Then, we have
\begin{equation}\label{e26}
\begin{aligned}
&\theta(u) - \theta(u^{k+1}) + (\omega - \omega^{k+1})^T F(\omega)\\
\geq& (z-z^{k+1})^T\beta(Dx^k-Dx^{k+1})+ \frac{1}{\beta}(\lambda - \lambda^{k+1})^T(\lambda^k - \lambda^{k+1}) , ~\forall\omega\in\Omega.
\end{aligned}
\end{equation}
\end{lemma}

\begin{proof}
From (\ref{e25}) we know that
\begin{equation}\label{e27}
\begin{aligned}
\theta_1(z) -  \theta_1(z^{k+1}) +(z - z^{k+1})^T [\lambda^k - \beta(Dx^k - z^{k+1})]\geq 0,~\forall z\in\mathcal{Z}
\end{aligned}
\end{equation}
and
\begin{eqnarray}\label{e28}
\begin{aligned}
\theta_2(x) -  \theta_2(x^{k+1}) + (x - x^{k+1})^T (-D^T \lambda^k + \beta D^T(Dx^{k+1} - z^{k+1}))\geq 0,~\forall x\in\mathcal{X}.
\end{aligned}
\end{eqnarray}
Using $\lambda^{k+1} = \lambda^{k} - \beta(Dx^{k+1} - z^{k+1})$ we can easily deduce
\begin{eqnarray}\label{e29}
\lambda^{k} = \lambda^{k+1} + \beta(Dx^{k+1} - z^{k+1})
\end{eqnarray}
and
\begin{eqnarray}\label{e30}
(Dx^{k+1} - z^{k+1}) = \frac{1}{\beta}(\lambda^{k} - \lambda^{k+1}).
\end{eqnarray}
Putting the formulations (\ref{e29}) and (\ref{e30}) into (\ref{e27}) and (\ref{e28}), respectively, then we have the following inequalities
\begin{equation}\label{e31}
\begin{aligned}
\theta_1(z) - \theta_1(z^{k+1}) + (z - z^{k+1})^T [\lambda^{k+1} + \beta(Dx^{k+1} - z^{k+1}) - \beta(Dx^k - z^{k+1})]\geq 0,
\end{aligned}
\end{equation}
\begin{eqnarray}\label{e32}
\theta_2(x) - \theta_2(x^{k+1}) + (x - x^{k+1})^T (- D^T\lambda^{k+1})\geq 0,
\end{eqnarray}
and
\begin{eqnarray}\label{e33}
\begin{aligned}
(\lambda - \lambda^{k+1})^T (Dx^{k+1} - z^{k+1}) \geq (\lambda - \lambda^{k+1})^T\frac{1}{\beta}(\lambda^{k} - \lambda^{k+1}).
\end{aligned}
\end{eqnarray}
Combining (\ref{e31}), (\ref{e32}) and (\ref{e33}) we have
\begin{eqnarray}\label{e34}
\begin{aligned}
\left\{
      \begin{array}{lll}
      \theta_1(z) - \theta_1(z^{k+1}) + (z - z^{k+1})^T \lambda^{k+1}  \geq (z - z^{k+1})^T\beta(Dx^k - Dx^{k+1}),\\
      \theta_2(x) - \theta_2(x^{k+1}) + (x - x^{k+1})^T (-D^T \lambda^{k+1}) \geq 0,\\
      (\lambda - \lambda^{k+1})^T (Dx^{k+1} - z^{k+1})  \geq (\lambda - \lambda^{k+1})^T\frac{1}{\beta}(\lambda^{k} - \lambda^{k+1}),
      \end{array}
\right.
\end{aligned}
\end{eqnarray}
which is
\begin{equation*}\label{e35}
\begin{aligned}
&\theta(u) - \theta(u^{k+1}) + (\omega - \omega^{k+1})^T F(\omega^{k+1})\\
\geq &  (z - z^{k+1})^T\beta(Dx^k - Dx^{k+1}) + (\lambda - \lambda^{k+1})^T\frac{1}{\beta}(\lambda^{k} - \lambda^{k+1}).
\end{aligned}
\end{equation*}
Note that the matrix in the operator $F$ is skew-symmetric, then, using (\ref{e36}), we have
\begin{equation}\label{e37}
\begin{aligned}
&\theta(u) - \theta(u^{k+1}) + (\omega - \omega^{k+1})^T F(\omega)\\
\geq &  (z - z^{k+1})^T\beta(Dx^k - Dx^{k+1})+ (\lambda - \lambda^{k+1})^T\frac{1}{\beta}(\lambda^{k} - \lambda^{k+1}).
\end{aligned}
\end{equation}
The Lemma \ref{L1} is proved.
\end{proof}

\begin{lemma}\label{L2}
Let the sequence $\{\omega^k\}$ be generated by (\ref{e24}). Then, we have
\begin{equation}\label{e38}
\begin{aligned}
&\beta(z - z^{k+1})^T  (Dx^k - Dx^{k+1}) + \frac{1}{\beta}(\lambda - \lambda^{k+1})^T(\lambda^k - \lambda^{k+1})\\
= &-\frac{1}{2\beta}\|\lambda^k - \lambda\|_2^2 - \frac{\beta}{2}\|Dx^k - z\|_2^2 +  \frac{1}{2\beta}\|\lambda^{k+1} - \lambda\|_2^2 + \frac{\beta}{2}\|Dx^{k+1} - z\|_2^2\\
&+ \frac{\beta}{2}\|Dx^k - z^{k+1}\|_2^2.
\end{aligned}
\end{equation}
\end{lemma}

\begin{proof}
Applying the identity
\begin{equation*}
\begin{aligned}
(a - b)^T(c-d)= \frac{1}{2}\{\|a-d\|_2^2 - \|a-c\|_2^2\} + \frac{1}{2}\{\|c-b\|_2^2 - \|d-b\|_2^2\}
\end{aligned}
\end{equation*}
to the left-hand side in (\ref{e38}) with
$$
a=z,~b=z^{k+1},~c=Dx^k,~ d= Dx^{k+1},
$$
we obtain
\begin{equation}\label{e39}
\begin{aligned}
&\beta(z - z^{k+1})^T  (Dx^k - Dx^{k+1}) \\
= &\frac{\beta}{2}\{\|z - Dx^{k+1}\|_2^2 - \|z - Dx^k\|_2^2\} + \frac{\beta}{2}\{\|Dx^k - z^{k+1}\|_2^2 - \|Dx^{k+1} - z^{k+1}\|_2^2\}.
\end{aligned}
\end{equation}
Using the identity
$$
b^T(b - a) = \frac{1}{2}(\|b\|_2^2 - \|a\|_2^2 + \|b - a\|_2^2),
$$
and let
$$
a = \lambda - \lambda^k,~b= \lambda - \lambda^{k+1},
$$
we obtain
\begin{equation}\label{e40}
\begin{aligned}
\frac{1}{\beta}(\lambda - \lambda^{k+1})^T(\lambda^k - \lambda^{k+1})
= \frac{1}{2\beta}\{\|\lambda - \lambda^{k+1}\|_2^2 - \|\lambda - \lambda^k\|_2^2 + \|\lambda^k - \lambda^{k+1}\|_2^2\}.
\end{aligned}
\end{equation}
Using
$$
\beta\|Dx^{k+1} - z^{k+1}\|_2^2 = \frac{1}{\beta}\|\lambda^k - \lambda^{k+1}\|_2^2,
$$
and combining (\ref{e39}) and (\ref{e40}), we complete the proof of this lemma.
\end{proof}

\begin{lemma}\label{L3}
Let the sequence $\{x^k\}$, $\{z^k\}$ and $\{\lambda^k\}$ be generated by (\ref{e24}), then,
\begin{eqnarray}\label{e42}
\begin{aligned}
\beta\|Dx^k - z^{k+1}\|_2^2 \geq \beta\|Dx^k - Dx^{k+1}\|_2^2 +\frac{1}{\beta}\|\lambda^k - \lambda^{k+1}\|_2^2.
\end{aligned}
\end{eqnarray}
\end{lemma}

\begin{proof}
Based on the second inequality of inequality \eqref{e34}, we can derive the following result
\begin{eqnarray}\label{e43}
\left\{
      \begin{array}{ll}
       \theta_2(x) - \theta_2(x^{k+1}) + (x - x^{k+1})^T (-D^T \lambda^{k+1})\geq 0,\\
       \theta_2(x) - \theta_2(x^{k}) + (x - x^k)^T (-D^T \lambda^k)\geq 0.
      \end{array}
\right.
\end{eqnarray}
Let $x=x^k$ and $x=x^{k+1}$ in \eqref{e43}, respectively, then
\begin{eqnarray*}\label{e44}
\left\{
      \begin{array}{ll}
       \theta_2(x^k) - \theta_2(x^{k+1}) + (x^k - x^{k+1})^T (-D^T \lambda^{k+1})\geq 0,\\
       \theta_2(x^{k+1}) - \theta_2(x^{k}) + (x^{k+1} - x^k)^T (-D^T \lambda^k)\geq 0.\\
      \end{array}
\right.
\end{eqnarray*}
From above inequalities, we have
\begin{eqnarray}\label{e45}
(\lambda^k - \lambda^{k+1})^T (Dx^k - Dx^{k+1})\geq 0.
\end{eqnarray}
Using
$$
(Dx^{k+1} - z^{k+1}) = \frac{1}{\beta}(\lambda^k - \lambda^{k+1}),
$$
then we obtain
\begin{equation}\label{e46}
\begin{aligned}
&\beta\|Dx^{k} - z^{k+1}\|_2^2 \\
=& \beta\|Dx^k - Dx^{k+1} + Dx^{k+1} - z^{k+1}\|_2^2\\
=& \beta\|Dx^k - Dx^{k+1} + \frac{1}{\beta}(\lambda^k - \lambda^{k+1})\|_2^2\\
\geq& \beta\|Dx^k - Dx^{k+1}\|_2^2 + \frac{1}{\beta}\|\lambda^k - \lambda^{k+1})\|_2^2.
\end{aligned}
\end{equation}
The proof of this lemma is completed.
\end{proof}

\subsection{Strict contraction}

To present the main result of the paper, it is necessary to establish the strict contractility of the iterative sequence. The following subsection provides a proof of the strong contractility of the iterative sequence $\{\omega^k\}$, which relies on Lemma \ref{L1}, Lemma \ref{L2}, and Lemma \ref{L3}.
\begin{theorem}\label{T1}
Assuming that the sequence $\{\omega^k\}$ is generated by equation (\ref{e24}), we can state the following
\begin{eqnarray}\label{e47}
\|v^{k+1} - v^*\|_H^2 \leq \|v^{k} - v^*\|_H^2 - \|v^{k} - v^{k+1}\|_H^2
\end{eqnarray}
where
\begin{eqnarray}\label{e48}
\begin{aligned}
v = \left(
    \begin{array}{c}
      \lambda \\
      x \\
    \end{array}
  \right),~
~H = \left(
      \begin{array}{cc}
        \frac{1}{\beta}I_m & 0 \\
        0 &  \beta I_d\\
      \end{array}
    \right),
~~\mathcal{V}^* = \{(\lambda^*,x^*)|(z^*,x^*,\lambda^*)\in\Omega\}.
    \end{aligned}
\end{eqnarray}
\end{theorem}

\begin{proof}
We can deduce from Lemma \ref{L1} and Lemma \ref{L2} that
\begin{equation}\label{e41}
\begin{aligned}
&\theta(u^{k+1}) - \theta(u)  + (\omega^{k+1} - \omega)^T F(\omega)\\
\leq &\frac{1}{2\beta}\|\lambda^k - \lambda\|_2^2 + \frac{\beta}{2}\|Dx^k - z\|_2^2 - \frac{1}{2\beta}\|\lambda^{k+1} - \lambda\|_2^2 - \frac{\beta}{2}\|Dx^{k+1} - z\|_2^2\\
&- \frac{\beta}{2}\|Dx^k - z^{k+1}\|_2^2.
\end{aligned}
\end{equation}
By utilizing Lemma \ref{L3}, we can rewrite equation (\ref{e41}) as follows
\begin{equation}\label{e49}
\begin{aligned}
0\leq&\theta(u^{k+1}) - \theta(u^*)  + (\omega^{k+1} - \omega^*)^T F(\omega^*)\\
\leq& \frac{1}{2\beta}\|\lambda^k - \lambda^*\|_2^2 + \frac{\beta}{2}\|Dx^k - z^*\|_2^2 - \frac{1}{2\beta}\|\lambda^{k+1} - \lambda^*\|_2^2 - \frac{\beta}{2}\|Dx^{k+1} - z^*\|_2^2 \\
-  &\frac{1}{2\beta}\|\lambda^k - \lambda^{k+1}\|_2^2 - \frac{\beta}{2}\|Dx^k - Dx^{k+1}\|_2^2 .
\end{aligned}
\end{equation}
That is
\begin{equation}\label{e50}
\begin{aligned}
&\frac{1}{\beta}\|\lambda^{k+1} - \lambda^*\|_2^2 + \beta\|Dx^{k+1} - z^*\|_2^2\\
\leq& \frac{1}{\beta}\|\lambda^k - \lambda^*\|_2^2 + \beta\|Dx^k - z^*\|_2^2 - (\frac{1}{\beta}\|\lambda^k - \lambda^{k+1}\|_2^2 + \beta\|Dx^k - Dx^{k+1}\|_2^2).
\end{aligned}
\end{equation}
Let
\begin{eqnarray*}\label{e51}
Dx^* = z^*,~
v = \left(
    \begin{array}{c}
      \lambda \\
      x \\
    \end{array}
  \right)~\textrm{and}
~H = \left(
      \begin{array}{cc}
        \frac{1}{\beta}I_m & 0 \\
        0 &  \beta D^T D\\
      \end{array}
    \right),
\end{eqnarray*}
therefore, the left-hand side of inequality \eqref{e50} becomes
\begin{equation}\label{e50-1}
\begin{aligned}
&\frac{1}{\beta}\|\lambda^{k+1} - \lambda^*\|_2^2 + \beta\|Dx^{k+1} - z^*\|_2^2 \\
= &\left(
      \begin{array}{c}
        \lambda^{k+1} - \lambda^* \\
        x^{k+1} - x^* \\
      \end{array}
    \right)^T
\left( \begin{array}{cc}
          \frac{1}{\beta}I_m & 0 \\
          0 & \beta D^T D \\
          \end{array}
          \right)\left(
                   \begin{array}{c}
                     \lambda^{k+1} - \lambda^* \\
                     x^{k+1} - x^* \\
                   \end{array}
                 \right)\\
  =&(v^{k+1} - v^*)^T\left(
                                             \begin{array}{cc}
                                             \frac{1}{\beta}I_m & 0 \\
                                             0 & \beta D^T D \\
                                             \end{array}
                                             \right)(v^{k+1} - v^*)\\
                                             =&\|v^{k+1} - v^*\|_H^2.
\end{aligned}
\end{equation}
Likewise, the sum of the first two terms on the right-hand side of inequality \eqref{e50} is
\begin{equation}\label{e50-2}
\begin{aligned}
&\frac{1}{\beta}\|\lambda^k - \lambda^*\|_2^2 + \beta\|Dx^k - Dx^*\|_2^2 \\
= &\left(
      \begin{array}{c}
        \lambda^{k} - \lambda^* \\
        x^{k} - x^* \\
      \end{array}
    \right)^T \left(
                           \begin{array}{cc}
                           \frac{1}{\beta}I_m & 0 \\
                            0 & \beta D^T D \\
                           \end{array}
                           \right)\left(
      \begin{array}{c}
        \lambda^{k} - \lambda^* \\
        x^{k} - x^* \\
      \end{array}
    \right)\\
                         =&(v^{k} - v^*)^T\left(
                                                                \begin{array}{cc}
                                                                \frac{1}{\beta}I_m & 0 \\
                                                                0 & \beta D^T D \\
                                                                \end{array}
                                                                \right)(v^{k} - v^*)\\
                                                              =&\|v^{k} - v^*\|_H^2\\
\end{aligned}
\end{equation}
and the sum of the last two terms on the right-hand side of inequality \eqref{e50} is
\begin{equation}\label{e50-3}
\begin{aligned}
&\frac{1}{\beta}\|\lambda^{k} - \lambda^{k+1}\|_2^2 + \beta\|Dx^{k} - Dx^{k+1}\|_2^2 \\
= &\left(
     \begin{array}{c}
       \lambda^{k} - \lambda^{k+1} \\
       x^{k} - x^{k+1} \\
     \end{array}
   \right)^T \left(
                         \begin{array}{cc}
                         \frac{1}{\beta}I_m & 0 \\
                         0 & \beta D^T D \\
                         \end{array}
                         \right)\left(
     \begin{array}{c}
       \lambda^{k} - \lambda^{k+1} \\
       x^{k} - x^{k+1} \\
     \end{array}
   \right)\\
   =&(v^{k} - v^{k+1})^T\left(
    \begin{array}{cc}
    \frac{1}{\beta}I_m & 0 \\
     0 & \beta D^T D \\
    \end{array}
    \right)(v^{k} - v^{k+1})\\
    =&  \|v^{k} - v^{k+1}\|_H^2.
\end{aligned}
\end{equation}
Since $D\in R^{n\times d}$ is a unit tight frame, we have that $D^T D = I_d$. By combining formulas \eqref{e50-1}, \eqref{e50-2}, and \eqref{e50-3}, we complete the proof of the Theorem \ref{T1}.
\end{proof}

According to Theorem \ref{T1}, we know that $H$ is a positive definite matrix, and inequality \eqref{e47} implies that the sequence $\{v^k\}$ is bounded. Assuming that the initial vector is $v_0 = (\lambda_0, x_0)^T$, we can obtain the following expression by summing both sides of inequality \eqref{e47}
\begin{equation}\label{e47-1}
\sum^{\infty}_{k=0}\|v^{k} - v^{k+1}\|_H^2 \leq \|v^{0} - v^{*}\|_H^2.
\end{equation}
The above equation indicates that $\lim_{k\rightarrow\infty} \|v^{k} - v^{k+1}\|_H^2 = 0$. Therefore, any subsequence ${v^{k_j}}$ of ${v^{k}}$ also has $\lim_{j\rightarrow\infty}\|v^{k_j} - v^{k_j+1}\|_H^2 = 0$. Suppose there exists a subsequence that converges to $\bar{v}$, then formula \eqref{e26} implies that $\bar{v}$ is the solution of formula \eqref{e24}. This shows that any accumulation point of the sequence ${v^{k}}$ is a solution of \eqref{e24}. According to formula \eqref{e47}, ${v^{k}}$ cannot have more than one accumulation point, and hence ${v^{k}}$ converges to $\bar{v}\in\mathcal{V}^*$.

\subsection{Convergence rate in ergodic sense}

Combining with Theorem \ref{T1}, we prove a worst-case $\mathcal{O}(1/t)$ convergence rate in a ergodic sense of the ADMM scheme \eqref{e24} for cosparse signal reconstruction problem.

\begin{theorem}\label{T2}
Let the sequence $\{\omega^k\}$ be generated by (\ref{e24}). Then, for any positive integer $t$, we have
\begin{eqnarray}\label{e52}
\begin{aligned}
&\theta(u^t) - \theta(u)  + (\omega^t - \omega)^T F(\omega)\\
\leq & \frac{1}{2(t+1)}[\frac{1}{\beta}\|\lambda^0 - \lambda\|_2^2 + \beta\|Dx^0 - z\|_2^2],~\forall \omega\in\Omega
\end{aligned}
\end{eqnarray}
where
\begin{eqnarray}\label{e53}
\omega^t = \frac{1}{t+1}(\sum_{k=0}^{t}\omega^{k+1}).
\end{eqnarray}
\end{theorem}

\begin{proof}
For any integer $k$, by (\ref{e41}) we obtain
\begin{equation}\label{e54}
\begin{aligned}
&\theta(u^{k+1}) - \theta(u)  + (\omega^{k+1} - \omega)^T F(\omega) \\
\leq & \frac{1}{2\beta}\|\lambda^k - \lambda\|_2^2 + \frac{\beta}{2}\|Dx^k - z\|_2^2 -  \frac{1}{2\beta}\|\lambda^{k+1} - \lambda\|_2^2
- \frac{\beta}{2}\|Dx^{k+1} - z\|_2^2.
\end{aligned}
\end{equation}
Suppose $k=0,1,2,\ldots, t$ are non-negative integers. By summing the left and right ends of the inequality \eqref{e54}, we deduce that
\begin{equation}\label{e55}
\begin{aligned}
&\sum_{k=0}^{t}\theta(u^{k+1}) - (t+1)\theta(u)  + \left[\sum_{k=0}^{t}\omega^{k+1} - (t+1)\omega\right]^T F(\omega) \\
\leq  & \frac{1}{2\beta}\|\lambda^0 - \lambda\|_2^2 + \frac{\beta}{2}\|Dx^0 - z\|_2^2,~~\forall\omega\in\Omega.
\end{aligned}
\end{equation}
The left and right ends of the inequality \eqref{e55} are multiplied by $\frac{1}{t+1}$ at the same time, and let
\begin{eqnarray}\label{e58}
\omega^t = \frac{1}{t+1}\sum_{k=0}^{t}\omega^{k+1},
\end{eqnarray}
then, the inequality \eqref{e55} is equivalent to
\begin{equation}\label{e56}
\begin{aligned}
&\frac{1}{t+1}\sum_{k=0}^{t}\theta(u^{k+1}) - \theta(u)  + (\omega^t - \omega)^T F(\omega) \\
\leq& \frac{1}{2(t+1)}[\frac{1}{\beta}\|\lambda^0 - \lambda\|_2^2 + \beta\|Dx^0 - z\|_2^2].
\end{aligned}
\end{equation}
Given that the function $\theta(u)$ is convex, let
\begin{equation*}\label{e56-1}
  u^t=\frac{1}{t+1}\sum_{k=0}^{t}u^{k+1} = \frac{1}{t+1}(u^1 + u^2 + \cdots + u^t),
\end{equation*}
we can derive the following expression
\begin{equation}\label{e57}
\begin{aligned}
\theta(u^t) =& \theta\left[ \frac{1}{t+1}(u^1 + u^2 + \cdots + u^t)\right]\\
\leq&\frac{1}{t+1}\left[\theta(u^1) + \theta(u^1) + \cdots + \theta(u^t)\right]\\
=&\frac{1}{t+1}\sum_{k=0}^{t}\theta(u^{k+1}).
\end{aligned}
\end{equation}
By utilizing equations \eqref{e56} and \eqref{e57}, we complete the proof of the Theorem \ref{T2}.
\end{proof}

After $t$-th iterations, then $\omega^t$ defined by (\ref{e58}) satisfies
\begin{eqnarray*}\label{e59}
\begin{aligned}
\tilde{\omega}\in\Omega ~~\textrm{and}~\sup_{\omega\in\mathcal{D}_{\tilde{\omega}}}\{\theta({\tilde{u}}) - \theta(u)  + (\tilde{\omega} - \omega)^T F(\omega)\}
\leq\frac{d}{2t}=\mathcal{O}(\frac{1}{t}),
\end{aligned}
\end{eqnarray*}
where
\begin{eqnarray*}\label{e60}
D_{\tilde{\omega}} = \{\omega\in \Omega | \|\omega - \tilde{\omega}\|\leq 1\},~~d: = \sup\{\frac{1}{\beta}\|\lambda^0 - \lambda\|_2^2 + \beta\|Dx^0 - z\|_2^2 | \omega\in\mathcal{D}_{\tilde{\omega}}\},
\end{eqnarray*}
and
$v^0=(\lambda^0,x^0)$ is the initial iteration point.
That means $\omega^t$ is an $\mathcal{O}(\frac{1}{t})$ solution of (\ref{e25}).

\section{Conclusions}
This paper presents a novel approach to analyze the convergence of cosparse optimization problem. In order to complete the proof of the main theorem, we first give the overall framework of the ADMM to solve the cosparse optimization problem, secondly, through three lemmas, this paper gives the basic inequalities required for the theorem proof, and finally, our analysis establishes a worst-case convergence rate of $\mathcal{O}(1/t)$, which demonstrates the effectiveness of our approach.


Researchers currently rely on a range of methods to solve separable convex optimization problems.
Two popular approaches are the generalized symmetric ADMM and parameterizable proximal point algorithms \cite{Han1,c68,Alves}.
These methods have demonstrated their effectiveness and superiority in various experiments. In our future work, we plan to explore the potential of combining these methods to solve the cosparse signal reconstruction problem.

\section*{Funding}
The authors were supported by the National Natural Science Foundation of China Mathematics Tian Yuan Fund under grant No. 12226323 and 12226315, the National Natural Science Foundation of China under grant No. 62103136, the Henan Province Undergraduate College Youth Backbone Teacher Training Program.

\section*{Acknowledgments}

The authors wish to thank Professor Zheng-Hai Huang for providing his valuable comments which have significantly improved the quality of this paper.

\end{document}